\documentclass{article}

\usepackage[letterpaper,top=2cm,bottom=2cm,left=3cm,right=3cm,marginparwidth=1.75cm]{geometry}

\usepackage{amsmath}
\usepackage{graphicx}
\usepackage[colorlinks=true, allcolors=blue]{hyperref}
\usepackage{graphicx}
\usepackage{natbib}
\usepackage{amsmath}
\usepackage{amssymb}

\usepackage{amsthm}
\theoremstyle{plain}% default
\newtheorem{thm}{Theorem}[section]

\newtheorem{proposition}[thm]{Proposition}
\theoremstyle{definition}

\theoremstyle{remark}
\newtheorem{rem}{Remark}

\usepackage{multirow}

\title{Goodness--of--fit tests for stochastic frontier models \\  based on the characteristic function}

\author{S. G. Meintanis\footnotemark[1] \;,\;  Christos K. Papadimitriou\footnote{Department of Economics, National and Kapodistrian University of Athens, Athens, Greece \newline  \textit{Corresponding author:} S. G. Meintanis (simosmei@econ.uoa.gr)} }

\date{}

\begin{document}
\maketitle

\begin{abstract}

We consider goodness--of--fit tests for the distribution of the composed error in Stochastic Frontier Models. The proposed test statistic utilizes the characteristic function of the composed error term, and is formulated as a weighted integral of properly standardized data. The new test statistic is shown to be consistent and computationally convenient. Simulation results are presented whereby  resampling versions of the new tests are compared to classical goodness--of--fit methods.
%\keywords{Goodness--of--fit \and Parametric Bootstrap \and Maximum Likelihood Estimation \and Normal/Exponential Distribution \and Normal/Gamma Distribution} 
%\PACS{C10 \and C12 \and C46 \and C52}

\end{abstract}

\section{Introduction}

\label{intro}
The stochastic frontier production model (SFM) was first introduced by Aigner, Lovell, and Schmidt (ALS) \citeyearpar{ALS77} and Meeusen and van den Broeck \citeyearpar{MB77}, in the form of a  Cobb-Douglas production function,
\begin{equation} \label{SFM}
Y=\beta^\top X+\varepsilon=\beta^\top X+v-u,  
\end{equation}
where $Y$ is the maximum log--output obtainable from a vector of log--inputs $X=(x_1,...,x_d)\in \mathbb R^d$, $\beta\in\mathbb R^d,  \ (d\geq 1)$, is an unknown vector of parameters, and $\varepsilon=v-u$, denotes the composed error term. The random component $v$ is intended to capture the effects of purely random statistical noise (disturbances beyond the firm's control), while $u\geq0$ is intended to capture the effects of technical inefficiency which are specific to each firm.

We now review earlier work on goodness--of--fit tests for certain aspects of the SFM. Schmidt and Lin \citeyearpar{SL84} and Coelli \citeyearpar{Co95} suggested tests of normality for the composed error term $\varepsilon$ by means of the empirical third moment of the OLS residuals. A moment-based method employing skewness and excess kurtosis has also been proposed quite recently by Papadopoulos and Parmeter \citeyearpar{PP21}. Lee \citeyearpar{Lee83} proposed Lagrange Multiplier tests for the normal/half-normal and the normal/truncated-normal SFM  within the Pearson family of truncated distributions. Kopp and Mullahy \citeyearpar{KM90} introduced GMM--based tests for the distribution of the inefficiency component $u$ by simply assuming the noise component $v$ to be  symmetric, but not necessarily normally distributed. They also suggest a GMM--based test for the symmetry assumption utilizing odd order moments of residuals. Bera and Mallick \citeyearpar{BeMa02} also suggest tests that enjoy moment interpretations but they test their moment restrictions by means of the information matrix. Most of the aforementioned tests however are not omnibus, i.e. they may have negligible power against certain alternatives. While Wang et al. \citeyearpar{WAS11} also suggest certain non-omnibus procedures, they  are probably the first to apply specification procedures, such as the Kolmogorov--Smirnov test, that are {\it{omnibus}}, i.e.  procedures which, being based on consistent tests,  enjoy non-negligible power for arbitrary deviations from the null model, and not just for directive alternatives. These authors  are also innovative by suggesting the use of the {\it{bootstrap}} in order to compute critical points and actually carry out the test in practice.  A further innovation is brought forward by Chen and Wang \citeyearpar{CW12} who in effect propose to use the {\it{characteristic function}} (CF) for testing distributional specifications in SFMs.
 
One important aspect of the SFM specification methods is the distribution of the technical inefficiency term $u$. In this paper we proceed in the lines set forward by Wang et al. \citeyearpar{WAS11} and Chen and Wang \citeyearpar{CW12} and suggest bootstrap--based omnibus specification tests for the composed error with special emphasis on the law of the inefficiency component  $u$ in  SFMs that utilize the CF. Our tests make use of the fact that  CFs are often easier to compute than densities or distribution functions, and also utilize the property that conditionally on the independence of  $u$ and $v$, the CF of the composed error term $\varepsilon$ may easily be obtained from the product of the CFs of $u$ and $v$. The rest of the paper unfolds as follows. In Section \ref{sec_2} we introduce the tests, discuss some aspects of the test statistics, prove consistency, and also consider estimation of parameters. In Section \ref{sec_3} we present a Monte Carlo study of a bootstrap--based version of the news  tests in the case of a normal/exponential and a normal/gamma SFM.  Conclusions and outlook are presented in Section \ref{sec_4}. A few technical arguments are deferred to Appendices \ref{Appendix A} and \ref{Appendix B}. There is also an accompanying \textit{Supplement} containing Monte Carlo results for some extra simulation settings.

%%%%%%%%%%%%%% SECTION 2 %%%%%%%%%%%%%%%%%%%%%%
\section{Goodness--of--fit tests}\label{sec_2}
In this section we consider tests for SFMs with exponentially distributed inefficiency and also tests for SFMs with gamma distributed inefficiency. In the first case the parameters are fully unspecified while in the latter case they are partially specified.  
\subsection{Tests for the composed error with exponential inefficiency}
Let $Z$ denote an arbitrary random variable, and recall that the CF of $Z$ is defined by $\varphi_Z (t)=\mathbb{E}({e^{itZ} })=\mathbb E[\cos(t Z)+i \sin(tZ)]\equiv C_Z(t)+iS_Z(t), \: t\in \mathbb R$, with $i=\sqrt{-1}$, where $C_Z(t)$ and $S_Z(t)$ denote the real and imaginary parts, respectively, of $\varphi_Z(t)$.  A few basic properties of CFs that will be used here are the following: (a) For the CF of $Z$ it holds $\varphi_Z(-t)=\varphi_{-Z}(t)={\overline {\varphi_Z(t)}}$, where $\overline z$ denotes the conjugate of the complex number $z$, (b) if $Z$ has a symmetric around zero distribution, then its CF is real-valued, i.e. it holds $S_Z(t)\equiv 0$, and hence $\varphi_Z(t)\equiv C_Z(t)$, and (c) if $Z_1$ and $Z_2$ are independent then $\varphi_{Z_1+Z_2}(t)=\varphi_{Z_1}(t)\varphi_{Z_2}(t)$.

Consider now the SFM in equation \eqref{SFM}, and suppose that on the basis of data $(X_j,Y_j), \ j=1,...,n$, we wish to test the null hypothesis
\begin{equation} \label{null}
\begin{split}
{\cal{H}}_{0}: \; & \textrm{Model (\ref{SFM}) holds true with }  
u\sim {\rm{Exp}}(\theta) \textrm{ , for some } \theta >0,
\end{split}
\end{equation}
where ${\rm{Exp}}(\theta)$ denotes the exponential distribution with density $\theta^{-1} \ e^{-x/\theta}$. At this stage the law of the pure statistical error $v$ will be left unspecified.  

In this connection, and since in the context of SFMs, $u$ and $v$ are assumed independent and $v$ is typically assumed to have a distribution that is  symmetric around zero, it readily follows that the CF of the composed error term may be computed as 
\begin{eqnarray} \label{cf1}
\varphi_{\varepsilon}(t)& = & \varphi_{v-u}(t)=\varphi_v(t)\varphi_u(-t)= C_v(t){\overline {\varphi_u(t)}}\\ \nonumber &=&  C_v(t)(C_u(t)- iS_u(t)). \end{eqnarray}
and hence if in addition we assume that $C_v(t)\neq 0, \ t\in \mathbb R$, then
\begin{equation} \label{cf11}
S_{\varepsilon}(t) + tC_{\varepsilon}(t)=-C_v(t)\left(S_u(t)-t C_u(t)\right)=0,  
\end{equation}
if and only if 
\begin{equation}
\label{cf1u}
S_{u}(t) - tC_{u}(t)=0. 
\end{equation}
However, Henze and Meintanis \citeyearpar{HM02} have shown that equation \eqref{cf1u} is a characterization of the  unit exponential distribution. Consequently  
\eqref{cf11} holds, if and only if \eqref{cf1u} holds for all $t\in\mathbb R$, which in turn only holds  under the null hypothesis ${\cal{H}}_{0}$ in \eqref{null} with $\theta=1$. This fact justifies taking the left hand side of equation \eqref{cf11} as the point of departure of our test, but in order to reduce the test of the null hypothesis ${\cal{H}}_0$ to unit exponentiality, we will consider instead of $\varepsilon$ the standardize error defined by  $\widetilde \varepsilon=\varepsilon/\theta$.

To this end recall that the SFM in \eqref{SFM} depends on the regression parameter $\beta$, that under the null hypothesis \eqref{null} this model also involves the exponential parameter $\theta$ and that the pure statistical error $v$ may also involve an unknown parameter. Let  $\widehat{\varepsilon}_j=Y_j-\widehat\beta^\top X_j$, be the residuals of the SFM (\ref{SFM}) under the null hypothesis ${\cal{H}}_0$. Clearly these residuals, besides being dependent on the regression parameter $\widehat \beta$, they are also computed conditionally on suitable estimates of the aforementioned distributional parameters; see Section \ref{sec_3} for parameter estimation. We will write  $\widehat{\widetilde{\varepsilon}}_j=\widehat{\varepsilon}_j/\widehat{\theta}, \ j=1,\ldots,n$, for the respective standardized residuals;  Then the left hand side of equation (\ref{cf11}) may be estimated by 
\begin{equation} \label{cf12}
D_n(t):=S_n(t)+tC_n(t),  t \in \mathbb{R}, 
\end{equation}
where 
\begin{displaymath}
C_n(t)=\frac{1}{n} \sum_{j=1}^{n}\cos(t\widehat{\widetilde\varepsilon}_j), \ S_n(t)=\frac{1}{n} \sum_{j=1}^{n}\sin(t\widehat{\widetilde\varepsilon}_j),
\end{displaymath}
with $C_n$ (resp. $S_n$) being an estimator of 
$C_{\widetilde\varepsilon}$ (resp. $S_{\widetilde\varepsilon}$). In view of \eqref{cf11}, $D_n(t)$ should be close to zero under the null hypothesis ${\cal{H}}_0$ identically in $t\in \mathbb R$, at least for large sample size $n$. Thus it is reasonable to reject ${\cal{H}}_0$ for large values of the test statistic
\begin{equation} \label{ts}
T_{n,w}= n \int_{-\infty}^{\infty} D^2_{n}(t) \; w(t) \; {\rm{d}}t,   
\end{equation}
where $w(t)>0$ is an integrable weight function\footnote{The test for the cost frontier model with $\varepsilon=v+u$ may be computed by modifying equation \eqref{cf12} to $D_n(t):=S_n(t)-tC_n(t)$ and by analogously defining the test statistic in equation \eqref{ts}.}. The test figuring in \eqref{ts} is an integrated and weighted test, and in this sense it is analogous to a Cram\'er-von Mises test, the only difference being that instead of the estimated distribution function used in the latter test, within $T_{n,w}$, and specifically in $D^2_{n}(t)$, we employ the estimated CF of the underlying law. In view of the uniqueness property of CFs and the positivity of the weight function $w(t)$,  formulation \eqref{ts} leads to a test statistic that is (globally) consistent, and thus to an omnibus test; see Proposition \ref{prop1}. However the uniqueness property of CFs only holds if, as it is done in \eqref{ts},  this CF is considered over all possible arguments $t \in \mathbb R$, and therefore the chi-squared tests suggested in Chen and Wang \citeyearpar{CW12} (see also Wang et al. \citeyearpar{WAS11}) which are based on computing the empirical CF over a finite grid of points, are not omnibus. One the other hand chi-squared tests have the advantage of a simple limit null distribution with well known and tabulated critical points, and thus the practitioner does not necessarily need to resort to bootstrap resampling. Despite this advantage however, convergence to the limit chi-squared distribution may prove quite slow and thus often one has to revert back to bootstrap for actual test implementation; see for instance Wang et al. \citeyearpar{WAS11}.

While the consistency of the test based on $T_{n,w}$  may be proved for a general class of weight functions, the choice $w(t)=e^{-\lambda |t|}, \ \lambda>0$, is particularly appealing from the computational point of view. To see this write $T_{n,\lambda}$ for the test statistic in \eqref{ts} with weight function $e^{-\lambda |t|}$. Then after some straightforward algebra (refer to Appendix \ref{Appendix A} for details) we obtain 
\begin{eqnarray} \label{ts2}
T_{n,\lambda}  & = &  \nonumber
\frac{\lambda}{n}\sum_{j,k=1}^n \frac{1}{\lambda^2+(\widehat{\widetilde{\varepsilon}}_{jk}^{-})^2} -
\frac{1}{\lambda^2+(\widehat{\widetilde{\varepsilon}}_{jk}^{+})^2} + \frac{4\widehat{\widetilde{\varepsilon}}_{jk}^{+}}{(\lambda^2+(\widehat{\widetilde{\varepsilon}}_{jk}^{+})^2)^2}  \\ 
 & + & \; 
 \frac{2(\lambda^2-3(\widehat{\widetilde{\varepsilon}}_{jk}^{+})^2)}{(\lambda^2+(\widehat{\widetilde{\varepsilon}}_{jk}^{+})^2)^3} + \frac{2(\lambda^2-3(\widehat{\widetilde{\varepsilon}}_{jk}^{-})^2)}{(\lambda^2+(\widehat{\widetilde{\varepsilon}}_{jk}^{-})^2)^3},
\end{eqnarray}
where we write $\sum_{j,k}$ for the double sum $\sum_j\sum_k$, and where $\widehat{\widetilde{\varepsilon}}_{jk}^{+}=\widehat{\widetilde{\varepsilon}}_{k}+\widehat{\widetilde{\varepsilon}}_{j}$, and $\widehat{\widetilde{\varepsilon}}_{jk}^{-}=\widehat{\widetilde{\varepsilon}}_{k}-\widehat{\widetilde{\varepsilon}}_{j}$, $j,k=1,\ldots,n$. 

We now illustrate the role that the weight function $e^{-\lambda |t|}$ plays in the test statistic $T_{n,\lambda}$. To this end we use  expansions 
of the trigonometric functions $\sin(\cdot)$ and $\cos(\cdot)$, and after some algebra (refer to Appendix \ref{Appendix B} for details) we obtain from \eqref{ts}   
%\begin{equation*}
%T_{n,\lambda}=n\left(\frac{1}{n} \sum_{j=1}^n \widehat{\widetilde{\varepsilon}}_{j}+1\right)^2 \int_{-\infty}^\infty t^2 e^{-\lambda |t|} dt + %{\rm{o}}(\lambda^{-3}), \end{equation*}
%$\lambda \to \infty$, which entails
\begin{eqnarray}\label{limstat}
\lim_{\lambda \to \infty} \frac{\lambda^3}{4} T_{n,\lambda} = n\left(\frac{1}{n} \sum_{j=1}^n \widehat{\widetilde{\varepsilon}}_{j}+1\right)^2. 
\end{eqnarray}

The ``limit statistic" in the right--hand side of \eqref{limstat} measures normalized distance from unity of the sample mean of the standardized residuals  $\widehat{\widetilde{\varepsilon}}_{j}$. Recall in this connection that under the null hypothesis ${\cal{H}}_0$, $\mathbb E(\widetilde\varepsilon)=-1$, and thus this distance should vanish under ${\cal{H}}_0$, as $n\to\infty$. However this same distance will also vanish under an alternative for which the standardized error term happens to have expectation equal to one. In conclusion taking a value of the weight parameter $\lambda$ that is ``too large" forces the test to depend on lower order moments of the residuals and should be avoided if the test is to have good power against alternatives with arbitrary moment structure. On the other hand, values of $\lambda$ too close to the origin result in a test that is prone to numerical error due to periodicity of  trigonometric functions.       

\begin{rem}\label{rem1} 
It is  clear  from  equations \eqref{cf1}--\eqref{cf11} which are instrumental in defining our test statistic that \eqref{cf1u} is robust to the law of the pure statistical error $v$, as long as this law satisfies $C_v(t)\neq 0, \ t \in \mathbb R$.  Of course the test statistic is  conditioned on a preliminary estimation step, and thus rejecting on the basis of the test in \eqref{ts} implies rejection of the ``entire" normal/exponential law for the composed error in that this entire law is present in the test statistic both at the estimation  as well as at the test construction step following the estimation step. In this sense  our test has power not only against non-exponential specifications for $u$, but also against any non-normal specification for $v$, such as the Student-t (see Wheat et al. \citeyearpar{WSG19}) and the stable (see Tsionas \citeyearpar{tsionas12}) specification. We refer the interested reader to the accompanying \textit{Supplement} for corresponding Monte Carlo results. In this connection we note that a large class of distributions  with $\varphi_Z\neq 0$ is the class of infinitely divisible laws (see Sasv\'ari \citeyearpar{sasvari}, \S3.11). At the same time, and while tailored specifically to the null hypothesis ${\cal{H}}_0$ of exponentiality in \eqref{null}, our test may also be applied with any other law of $v$ with a non-vanishing CF. To do so one has to apply \eqref{ts} as test statistic but the residuals  have to be computed via, say maximum likelihood, that takes into account the specific non-normal law postulated for $v$.
\end{rem}

Continuing on the power properties and due to the  uniqueness of CFs (see Sasv\'ari \citeyearpar{sasvari} Theorem 1.3.3), we maintain that the test statistic $T_{n,w}$ defined by \eqref{ts} has asymptotic power one as $n\to\infty$ for arbitrary deviations from the null hypothesis ${\cal{H}}_0$. This result is formally stated and proved below. As it is already implicit $X^\top$ denotes vector transposition, and we also write $\|X\|=(\sum_{k=1}^d x^2_k)^{1/2}$ for the Euclidean length of $X$.      

\begin{proposition}\label{prop1}
Consider the SFM in \eqref{SFM} and suppose the following conditions hold: {\rm{(C1)}} The CF of $v$ is real--valued and satisfies $\varphi_v\neq 0$,  the regressor $X\in \mathbb R^d, \ d\geq 1$, has finite mean and $(X,v,u)$ are mutually independent, {\rm{(C2)}} the distributions of $u$ and $X$ are such that $u+a^\top X$ is not exponentially distributed for any $d$--vector $a\neq 0$ {\rm{(C3)}} the estimator  $\widehat \beta$  satisfies $\widehat\beta\to b$ almost surely (a.s.) as $n\to\infty$ for some $b\in\mathbb R^d$, with $b=\beta_0$ (the true value)  under ${\cal{H}}_0$ and {\rm{(C4)}}  the weight function $w>0$ is such that $\int_{\mathbb R} t^2 w(t) dt<\infty$. Then for the test statistic in \eqref{ts} it holds 
\begin{equation} \label{conv}
\frac{T_{n,w}}{n} \rightarrow \int_{-\infty}^\infty D^2(t) w(t) dt:=\Delta_w,
\end{equation} 
a.s. as $n\to\infty$, with $D(t)=S_{e}(t)+tC_{e}(t)$, where $C_{e}$ (resp. $S_e$) denotes the real (resp. imaginary) part of the CF of 
$e_j(b):=Y_j-b^\top X_j, \ j=1,...,n$.    
\end{proposition}

\begin{proof} For simplicity we assume the distributional parameters to be fixed under the null hypothesis and that specifically $\theta=1$. Then the following Taylor expansion of the cosine function around $\beta=b$
\[
\cos(t e_j(\widehat \beta))=\cos(t e_j(b))+(\widehat \beta-b)^\top  \nabla  \cos(t e_j(\beta))\bigg|_{\beta=b^*}
\]
where
\[
\nabla  \cos(t e_j(\beta))= \left(\frac{\partial \cos(t e_j(\beta)) }{\partial \beta_1},..., \frac{\partial \cos(t e_j(\beta)) }{\partial \beta_d}\right)^\top
\]
and $b^*$ is such that $\|b^*-b\|\leq \|\widehat \beta-b\|$, leads to  

\[
\left|C_n(t)-\frac{1}{n}\sum_{j=1}^n\cos\left(t e_j\right)\right|\leq |t| \sum_{k=1}^d |\widehat\beta_k-b_k|\frac{1}{n} \sum_{j=1}^n |X_{jk}| \rightarrow 0,
\]
 a.s. as $n\to\infty$, so that $C_n(t)\to C_{e}(t)$ and likewise $S_n(t) \to S_{e}(t)$. Thus $D^2_n(t)\to D^2(t)$, and since  $D^2_n(t)\leq (1+|t|)^2$, with $\int_{\mathbb R} (1+|t|)^2w(t)dt<\infty$ by (C4), we may invoke Lebesgue's theorem of dominated convergence (see Jiang \citeyearpar{jiang}, \S A.2.3) and the proof of \eqref{conv} is finished. 
 Clearly $\Delta_w>0$ unless $D(t)=0$ identically in $t$. Now write $e=Y-b^\top X=\varepsilon-a^\top X=v-(u+a^\top X)$, where $a=b-\beta_0$, so that by independence $\varphi_e(t)=\varphi_v(t)\varphi_{u+a^\top X}(-t)=\varphi_v(t)(C_{u+a^\top X}(t) - i S_{u+a^\top X}(t))$ and therefore since by (C1) $\varphi_v\neq 0$, $D\equiv 0$ holds if and only if $S_{u+a^\top X}(t)=tC_{u+a^\top X}(t),$ identically in $t$, which is an established characterization of the  exponential distribution; see Henze and Meintanis \citeyearpar{HM02} for tests based on this characterization, and Jammalamadaka and Taufer \citeyearpar{JT03} and Henze and Meintanis \citeyearpar{HM05} for reviews on testing for exponentiality.  However condition (C2) rules out this possibility unless $u$ follows an exponential distribution, in which case $b=\beta_0$ (or $a=0$), i.e. $\Delta_w=0$ only under the null hypothesis ${\cal{H}}_0$ figuring in \eqref{null}.  
 Thus $T_{n,w} \to \infty$ a.s. as $n\to\infty$ under alternatives and consequently the test which rejects ${\cal{H}}_0$ for large values of $T_{n,w}$ is consistent.  \end{proof}   %$\blacksquare$

\begin{rem}
Formally speaking, for fixed distribution of the regressor $X$ with CF $\varphi_X\neq 0$, condition ${\rm{(C2)}}$ is violated if $\varphi_u(t)=\left((1-it)\varphi_{a^\top X}(t)\right)^{-1}$. For the circumstances under which this violation is possible  to become more transparent assume that $X\equiv 1$, i.e. assume that the simple location SFM, $Y=\beta+\varepsilon$, holds. Then this condition reads as  $\varphi_u(t)= (1 - it)^{-1}e^{- i A t}$, $A=\sum_{k=1}^d a_k$ (the sum of the elements of the vector $a$), meaning that $u=Z-A$, with $Z$ exponentially distributed. If this happens  however, then we are not in line with the classical assumption that the support of the distribution of the inefficiency component $u$ is the non--negative real line. Note that $u\geq 0$ figures in all parametric specifications of inefficiency, while at the same time it is probably the only single assumption which is maintained even under non--parametric approaches to SFM such as those of Kumbhakar et al. \citeyearpar{K07}.         
\end{rem}

%%%%%%%%%%%%%%%%% SUBSECTION 2.2 %%%%%%%%%%%%%%%%%%%%%%
\subsection{Estimation for the normal/exponential case}
As already mentioned the parameters of any given SFM are considered unknown and thus they should be estimated from the data $(X_j,Y_j), \ j=1,...,n$. Here we will illustrate the estimation procedure on the assumption that $v\sim{\cal{N}}(0,\sigma^2_v)$, i.e. that $v$ follows a zero--mean normal distribution with variance $\sigma^2_v$. In this connection one of the most commonly used estimator is the  maximum likelihood estimator (MLE), which is known to be consistent and asymptotically efficient.   

In order to compute the normal/exponential likelihood function we note that the density for the composed error term $\varepsilon$, is given by (see Kumbhakar and Lovell \citeyearpar{KuLo2000}),
\begin{equation} \label{nexdensity}
f(\varepsilon)= \frac{1}{\theta}  {\rm{\Phi}} \left(- \frac{\varepsilon}{\sigma_v} - \frac{\sigma_v}{\theta} \right )  {\rm{exp}}\left( \frac{\varepsilon}{\theta} + \frac{\sigma_v^2}{2 \theta^2}\right)
\end{equation}
where ${\rm{\Phi}}(\cdot)$ denotes the standard normal distribution function. Based on this equation,  the log-likelihood function for the sample may be written  as
\begin{equation} \label{loglik}
\begin{split}
\log L(\beta,\sigma_{v}^{2},\theta) & =   -n\log\theta 
+ n \left(\frac{\sigma_v^2}{2\theta^2}\right) +\sum_{j=1}^n \log \; {\rm{\Phi}}\left(-\frac{\varepsilon_j}{\sigma_v} -\frac{\sigma_v}{\theta}\right) +\sum_{j=1}^n\frac{\varepsilon_j}{\theta}.    
\end{split}
\end{equation}

Since the log-likelihood function in \eqref{loglik} is non-linear, iterative computational methods are needed to be developed. To this end, a Matlab code was developed in which the unconstrained maximisation  of (\ref{loglik}) is done using the library function \emph{fminunc}. For the implementation of this program the \emph{Quasi-Newton} method is used, instead of the \emph{Newton-Raphson} method, since the latter requires the calculation of second partial derivatives. 

\begin{rem}
Although in our simulations we apply the tests based on the MLE due to its efficiency, it should be pointed out that our tests may be applied with any given estimator provided that this estimator is consistent under the null hypothesis, while under alternatives attains an almost sure stochastic limit. Then the conclusions of Prop. \ref{prop1}  still hold true. Specifically the classical method of moments (and generalized versions thereof) considered by Kopp and Mullahy (\citeyear{KM90}, \citeyear      {KM93}), satisfies these conditions. In fact even moment--type procedures  based on the CF such as those used by Chen and Wang (2012) for testing purposes may also be used for estimation, and also lead, via the test statistic figuring in \eqref{ts}, to a test that is consistent against arbitrary deviations from the null hypothesis.
\end{rem}

\subsection{Tests for the composed error with gamma inefficiency}

%\section{Extension of the test to the normal/gamma case} \label{sec_4}
In this section we consider the test for a SFM with gamma distributed inefficiency term; see for example Tsionas \citeyearpar{Tsionas}. By way of example we consider the null hypothesis 
\begin{equation} \label{null1}
\begin{split}
{\cal{H}}_{0}: \;& \textrm{Model \eqref{SFM} holds true with }  u\sim {\rm{Gamma}}(\kappa,\theta)  \textrm{ , for $\kappa=2$ and some } \
\theta >0,   
\end{split}
\end{equation}
i.e. we consider testing for a gamma distribution with shape parameter equal to $\kappa=2$, and unspecified value of $\theta$.   
Recall in this connection the density of the gamma distribution is $f(x;\kappa,\theta)=(x^{\kappa-1}/({\rm{\Gamma}}(\kappa) \theta^\kappa))e^{-x/\theta}$, with ${\rm{\Gamma}}(\kappa)=\int_{0}^{\infty} x^{\kappa-1}e^{-x} dx$.

Since the CF of a random variable $Z$ following the  gamma distribution is given by  $\varphi_Z (t)= (1-it\theta)^{-\kappa}$, and by analogous steps as in Section \ref{sec_2} it follows that for $\kappa=2$  the CF of the standardized composed error $\widetilde \varepsilon=\varepsilon/\theta$ satisfies 
\begin{equation} \label{cfgamma}
(1-t^2)S_{\widetilde \varepsilon}(t)+2tC_{\widetilde \varepsilon}(t) = 0, \ t\in \mathbb R
\end{equation}
and therefore suggest a test statistic analogous to \eqref{ts} with 
\begin{equation} \label{cf13}
D_n(t)=(1-t^2)S_n(t)+2tC_n(t),  t \in \mathbb{R}, 
\end{equation}
where $C_n$ and $S_n$ are defined in the same way as in \eqref{cf12} but now the residuals are estimated from the SFM $Y_j=\beta+\varepsilon_j$ under the normal/gamma null hypothesis \eqref{null1} with $\kappa=2$.

With some further algebra it follows that if we employ the same  weight function $w(t)=e^{-\lambda|t|}$, the test statistic is rendered in the following form which is convenient for computer implementation (refer to Appendix \ref{Appendix A} for details):
\begin{eqnarray} \label{tsgamma}
T_{n,\lambda} & = & \frac{\lambda}{n} \sum_{j,k=1}^n \nonumber \frac{1}{\lambda^2+(\widehat{\widetilde\varepsilon}_{jk}^{-})^2} - \frac{1}{\lambda^2+(\widehat{\widetilde\varepsilon}_{jk}^{+})^2} 
+  \frac{4(\lambda^2-3(\widehat{\widetilde\varepsilon}_{jk}^{-})^2)}{(\lambda^2+(\widehat{\widetilde\varepsilon}_{jk}^{-})^2)^3}  +  \frac{12(\lambda^2-3(\widehat{\widetilde\varepsilon}_{jk}^{+})^2)}{(\lambda^2+(\widehat{\widetilde\varepsilon}_{jk}^{+})^2)^3}  \\  \nonumber
& + & \frac{8\widehat{\widetilde\varepsilon}_{jk}^{+}((\lambda^2+(\widehat{\widetilde\varepsilon}_{jk}^{+})^2)^2-12(\lambda^2-(\widehat{\widetilde\varepsilon}_{jk}^{+})^2))}{(\lambda^2+(\widehat{\widetilde\varepsilon}_{jk}^{+})^2)^4}  
+  \frac{24(\lambda^4-10\lambda^2(\widehat{\widetilde\varepsilon}_{jk}^{-})^2+5(\widehat{\widetilde\varepsilon}_{jk}^{-})^4)}{(\lambda^2+(\widehat{\widetilde\varepsilon}_{jk}^{-})^2)^5}  \\ 
& - & \frac{24(\lambda^4-10\lambda^2(\widehat{\widetilde\varepsilon}_{jk}^{+})^2+5(\widehat{\widetilde\varepsilon}_{jk}^{+})^4)}{(\lambda^2+(\widehat{\widetilde\varepsilon}_{jk}^{+})^2)^5},  
\end{eqnarray}
with $\widehat{\widetilde\varepsilon}_{jk}^{+}$ and $\widehat{\widetilde\varepsilon}_{jk}^{-}$ $j,k=1,\ldots,n$, defined in exactly the same way as in \eqref{ts2}.

The comments made in Remark \ref{rem1} apply here too, i.e. the test defined in \eqref{tsgamma} may be used with any given law of $v$ with real and non-vanishing CF, be it normal or non-normal, but we reiterate that this generality is conditioned on a proper estimation step that takes into account the specific law of $v$ postulated.

Before closing this part we also wish to emphasize that the tests considered herein, and as far as the law of the inefficiency is concerned, are specific to the null hypotheses as stated, i.e. the test in \eqref{ts} is specific to the null hypothesis of exponentiality figuring in \eqref{null} while the test in \eqref{tsgamma} is specific to the null hypothesis as stated in \eqref{null1}, and that in both cases the tests are not directional aiming against a specific alternative, but rather they have power against arbitrary deviations from the corresponding null hypothesis. On the other hand these tests  may be appropriately modified to test a more general  null hypothesis such as testing for a gamma distribution with unspecified value of $\kappa$, or to test a separate family of distributions like the popular half-normal specification for the technical efficiency term $u$. In doing so however, one has to take into account the specific structure of the CF of $u$ under this particular specification and design the test analogously; see Section \ref{sec_4} for some extra discussion on this issue.

%%%%%%%%%%%%%%%%%%%%%% SECTION 3 %%%%%%%%%%%%%%%%%%%
\section{Simulations}\label{sec_3}
\subsection{Simulations for the normal/exponential case}
%\subsection{The normal/exponential case} \label{subsec_4.1}
In this section we present the results of Monte Carlo study for the new test statistic given by equation \eqref{ts2}.\footnote{For simulations we used Matlab software R2015a version.} Specifically under the null hypothesis we consider the normal/exponential SFM whereby $v\sim{\cal{N}}(0,1)$ and $u\sim {\rm{Exp}}(\theta)$ for $\theta=0.5, 1.0, 3.0, 5.0$, $8.0, 10.0$, while the power of the test is computed against a normal/half--normal alternative hypothesis with the same Gaussian component, and the half--normal scale parameter set equal to $\sigma_u=0.5, 1.0, 3.0, 5.0$, $8.0, 10.0$.

We also compare the results of the proposed test statistic with those obtained from the classical Kolmogorov -- Smirnov (KS) and Cram\'er--von Mises (CvM) tests. For ease of reference we report the equations defining the KS and CvM test statistics. To this end note that both these statistics utilize the empirical cumulative distribution function $\widehat{F}_n(\cdot)$ of the residuals $\widehat\varepsilon_j$ and the theoretical (assumed) cumulative distribution function  $F_{0j}:=F_0(\widehat{\varepsilon}_j;\widehat\sigma_v,\widehat\theta)$ under the null hypothesis ${\cal{H}}_0$.   The respective formulas are given by 
\begin{equation} \label{KS}
\begin{split}
\mathrm{KS}  = \max\{D^+,D^-\}  \textrm{ , where }  D^+=\max\limits_{1 \leq j \leq n}\Big\{ \frac{j}{n}-F_{0j} \Big\} \textrm{ and }  D^-=\max\limits_{1 \leq j \leq n}\Big\{ F_{0j}-\frac{j-1}{n}) \Big\}
\end{split}
\end{equation}
and
\begin{equation} \label{CM}
\mathrm{CvM}=\frac{1}{12n}+\sum_{j=1}^{n}{\bigg(F_{0j} -\frac{2j-1}{2n}\bigg)}^2,
\end{equation}
with the assumed normal/exponential cumulative distribution function $F_0(\cdot)$ being computed by numerical integration. %For recent developments in computing  the cumulative distribution function of the composed error see Amsler et al. \citeyearpar{AST19} and Amsler et al. \citeyearpar{AST21} that consider the popular case of a normal/half-normal SFM. 

We consider the simple location SFM in equation (\ref{SFM}), $Y_j=\beta+\varepsilon_j$, with $\beta$ estimated by MLE. The number of Monte Carlo replications is  $M=1,000$, with sample size $n=100,200,300,500$, and nominal level of significance $\alpha=5\%$. For the new test statistic $T_{n,\lambda}$ we consider $\lambda=0.5, 1.0, 2.0, 3.0, 4.0$ and $5.0$. 

Since however the parameters of the model are considered unknown, we employ a  {\it{parametric bootstrap}} version of the tests which resamples from the null distribution with estimated parameters, and thus the extra variation due to parameter estimation is taken into account in computing critical values of test statistics; see for instance Babu and Rao \citeyearpar{BaRa04} for theory of the parametric bootstrap. However, the implementation of a Monte Carlo simulation employing the parametric bootstrap  will potentially incur a great cost in computational time due to the nested iteration structure involved with attempting to evaluate a bootstrap procedure in a Monte Carlo. To alleviate this computational burden, the so-called ``warp speed'' bootstrap procedure  will be used to approximate the bootstrap critical value in the Monte Carlo study. This bootstrap procedure which has been put on a firm theoretical basis by  Giacomini et al. \citeyearpar{GPW13} and Chang and Hall \citeyearpar{CH15} capitalizes on the repetition inherent in the Monte Carlo simulation to produce bootstrap replications, rather than relying on a separate ``bootstrap loop''. More specifically we calculate the bootstrap test statistic for only one bootstrap sample (single double-resampling) for each of the $M$ Monte Carlo iterations, and ultimately obtain $M$ bootstrap sample test statistics at the end of the simulation. The steps in performing the warp-speed version of the parametric bootstrap are itemized below for the normal/exponential case. For the normal/gamma case, step (B3) needs to be modified in an obvious manner.

\begin{itemize}
\item [(B1)] Draw a Monte Carlo sample $\{Y^{(m)}_j,X^{(m)}_j\}, \ j=1,...,n$,  compute the estimator--vector $\widehat{\Theta}^{(m)}$, where $\widehat{\Theta}^{(m)}=(\widehat{\beta}^{(m)}, \widehat{\sigma}^{(m)2}_v, \widehat{\theta}^{(m)})$.
\item[(B2)] On the basis of $\widehat{\Theta}^{(m)}$ calculate the residuals $\widehat{\varepsilon}^{(m)}_j$ and  the corresponding test statistic $T_m=T(\widehat{\varepsilon}^{(m)}_1,\ldots,\widehat{\varepsilon}^{(m)}_n)$.
\item[(B3)] Generate {\it{i.i.d.}} bootstrap errors $\varepsilon_j^{(m)},\ j=1,\ldots,n$,  where $\varepsilon_j^{(m)}=v_j^{(m)}-u_j^{(m)}$, with $v_j^{(m)} \sim {\cal{N}}(0,\widehat{\sigma}_v^2)$ and $u_j^{(m)} \sim {\rm{Exp}} (\widehat{\theta})$, and independent.
\item[(B4)] Define the bootstrap observations $Y_j^{(m)}= \widehat{\beta}^{(m)}X_j+\varepsilon_j^{(m)}$, $j=1,...,n$.  
\item[(B5)] Based on $\{Y_j^{(m)},X_j\}$  compute the bootstrap estimator $\widehat{\Theta}_b^{(m)}=(\widehat{\beta}_b^{(m)}, \widehat{\sigma}_{b,v}^{2 (m)}, \widehat{\theta}_b^{(m)})$,  and the corresponding bootstrap residuals, say, $\widehat{\epsilon}_{j}^{(m)}$, $j=1,...,n$. 
\item[(B6)] Compute the test statistic $\widehat T_m:=T(\widehat{\epsilon}^{(m)}_{1},\ldots,\widehat{\epsilon}^{(m)}_{n})$, based on the bootstrap residuals. 
\item[(B7)] Repeat steps (B1)--(B6), for  $m=1,...,M$, leading to  test--statistic values $T_m$ and bootstrap statistic values $\widehat T_m, \ m=1,...,M$. 
\item[(B8)] Set the critical point equal to $\widehat T_{(M-\alpha M)}$, where $\widehat T_{(m)}$, $m=1,\ldots,M$, denote the order statistics corresponding to $\widehat{T}_{m}$, and $\alpha$ denotes the prescribed size of the test.
%\item[(B9)] Reject the null hypothesis ${\cal{H}}_{0}$ if $T>\widehat T_{(B-\alpha B)}$. 
\end{itemize}

\begin{table*}[h!] %ht    % Table 1
\centering
\small
\setlength{\tabcolsep}{7pt} 
\renewcommand{\arraystretch}{0.90} 
\caption{Size of the test for the normal/exponential null hypothesis at level of significance $\alpha$ and sample size $n$}
\label{expsize}
\begin{tabular}{| l  r | c c c c c c | c | c |}
\hline
$\alpha=5\%$ &  & \multicolumn{6}{|c|}{$T_{n,\lambda}$} & \quad KS \quad & \quad CvM \quad \\
\cline{3-8} & $n$ & $\lambda=0.5$ & $\lambda=1.0$ & $\lambda=2.0$  & $\lambda=3.0$ & $\lambda=4.0$ &$\lambda=5.0$ &  &   \\ \hline \hline

$\theta=0.5$
& 100 & 4.4 & 6.0 & 7.2 & 7.1 & 6.4 & 7.3 & 5.6 & 5.8 \\ 
& 200 & 4.5 & 4.6 & 4.8 & 4.7 & 4.6 & 4.9 & 3.7 & 5.0 \\ 
& 300 & 5.8 & 4.3 & 5.0 & 5.9 & 5.8 & 5.4 & 5.7 & 3.8 \\ 
& 500 & 4.8 & 4.4 & 5.3 & 5.8 & 4.8 & 5.0 & 7.2 & 7.2 \\ \hline

$\theta=1$ 
& 100 & 4.8 & 4.0 & 4.9 & 3.9 & 3.2 & 2.8 & 5.5 & 5.2 \\ 
& 200 & 6.9 & 6.0 & 3.8 & 3.4 & 2.3 & 2.3 & 6.0 & 6.5 \\ 
& 300 & 6.1 & 5.3 & 4.1 & 3.0 & 2.4 & 2.5 & 5.8 & 5.4 \\ 
& 500 & 3.9 & 3.5 & 5.2 & 3.1 & 2.5 & 2.5 & 3.9 & 4.8 \\ \hline

$\theta=3$ 
& 100 & 4.9 & 3.0 & 2.8 & 3.6 & 4.5 & 4.5 & 4.6 & 3.5 \\ 
& 200 & 5.5 & 3.5 & 4.4 & 5.3 & 4.7 & 5.1 & 6.0 & 4.7 \\ 
& 300 & 3.8 & 2.4 & 2.8 & 2.6 & 2.1 & 2.6 & 2.7 & 2.6 \\ 
& 500 & 4.1 & 4.2 & 4.6 & 4.5 & 4.6 & 4.9 & 3.6 & 4.8 \\ \hline
	               
$\theta=5$ 
& 100 & 4.1 & 3.6 & 3.6 & 4.9 & 5.3 & 5.1 & 4.3 & 3.8 \\ 
& 200 & 3.2 & 4.2 & 4.1 & 4.7 & 5.2 & 4.9 & 3.8 & 3.6 \\ 
& 300 & 5.9 & 3.8 & 4.1 & 5.1 & 5.3 & 5.5 & 4.9 & 4.1 \\ 
& 500 & 5.1 & 4.9 & 6.7 & 6.5 & 5.4 & 4.9 & 6.9 & 6.2 \\ \hline

$\theta=8$ 
& 100 & 4.8 & 4.1 & 5.4 & 6.0 & 5.7 & 5.9 & 5.1 & 6.1 \\ 
& 200 & 2.9 & 4.7 & 4.8 & 6.0 & 6.4 & 6.0 & 5.0 & 4.0 \\ 
& 300 & 7.3 & 6.2 & 6.0 & 6.3 & 5.9 & 6.2 & 6.1 & 6.5 \\ 
& 500 & 5.5 & 6.2 & 5.6 & 6.1 & 6.4 & 5.8 & 5.8 & 5.4 \\ \hline
	               
$\theta=10$ 
& 100 & 2.9 & 3.3 & 3.4 & 3.5 & 4.4 & 3.8 & 3.8 & 3.7 \\ 
& 200 & 3.5 & 5.0 & 4.4 & 4.5 & 4.4 & 4.1 & 4.6 & 4.6 \\ 
& 300 & 5.2 & 3.5 & 4.5 & 4.3 & 4.6 & 4.7 & 5.1 & 4.7 \\ 
& 500 & 5.2 & 6.2 & 5.7 & 6.3 & 6.2 & 5.9 & 5.5 & 5.5 \\ \hline

\multicolumn{10}{l}{\footnotesize Number of Monte Carlo iterations $M=1,000$, KS: Kolmogorov-Smirnov test, CvM: Cram\'er-von Mises test,} \\
\multicolumn{10}{l}{\footnotesize  normal: standard normal, exponential: ${\rm{Exp}}(\theta)$.} 
\end{tabular}
\end{table*}
 
\begin{table*}[h!]    % Table 2
\centering
\small
\setlength{\tabcolsep}{7pt} 
\renewcommand{\arraystretch}{0.90}
\caption{Power of the test against the  normal/half-normal alternative at level of significance $\alpha$ and sample size $n$}
\label{exppower1}
\begin{tabular}{| l  r | c c c c c c | c | c |}
\hline
$\alpha=5\%$ &  & \multicolumn{6}{|c|}{$T_{n,\lambda}$} & \quad KS \quad & \quad CvM \quad \\
\cline{3-8} & $n$ & $\lambda=0.5$ & $\lambda=1.0$ & $\lambda=2.0$  & $\lambda=3.0$ & $\lambda=4.0$ &$\lambda=5.0$ &  &   \\ \hline \hline

$\sigma_u=0.5$ 
& 100 & 4.9 & 8.1 & 7.9 & 7.8 & 7.7 & 8.3 & 5.9 & 5.4 \\ 
& 200 & 6.0 & 6.2 & 7.2 & 7.4 & 8.2 & 8.5 & 5.9 & 5.1 \\ 
& 300 & 3.9 & 5.3 & 6.4 & 6.5 & 7.5 & 8.6 & 5.5 & 5.3 \\ 
& 500 & 4.6 & 4.7 & 5.6 & 6.1 & 7.1 & 7.3 & 5.1 & 4.5 \\ \hline

$\sigma_u=1$ 
& 100 & 4.3 & 5.4 & 8.5 & 9.7 & 9.4 & 9.2 & 6.1 & 5.4 \\ 
& 200 & 4.0 & 4.8 & 7.1 & 7.3 & 8.6 & 9.0 & 5.7 & 5.7 \\ 
& 300 & 5.7 & 5.0 & 5.9 & 6.0 & 5.1 & 5.2 & 5.8 & 6.0 \\ 
& 500 & 4.1 & 4.6 & 4.9 & 5.8 & 6.6 & 6.3 & 5.4 & 6.2 \\ \hline

$\sigma_u=3$ 
& 100 & 6.3  & 7.1  & 7.2  & 5.6  & 5.4  & 4.5  & 8.8 & 10.0  \\ 
& 200 & 6.8  & 9.6  & 13.5 & 12.9 & 12.1 & 11.9 & 15.3 & 15.4  \\ 
& 300 & 9.3  & 15.4 & 22.3 & 21.3 & 21.4 & 21.1 & 18.6 & 23.4   \\ 
& 500 & 17.7 & 33.4 & 50.4 & 58.9 & 61.4 & 58.5 & 35.3 & 43.7   \\ \hline
	               
$\sigma_u=5$ 
& 100 &   7.3 &   9.3 & 18.5 & 19.8 & 19.2 & 19.1 & 16.6 & 22.2   \\ 
& 200 & 13.9 & 26.6 & 36.8 & 40.7 & 41.6 & 41.9 & 31.7 & 38.4   \\ 
& 300 & 21.9 & 42.7 & 60.0 & 64.7 & 68.3 & 67.8 & 47.8 & 58.7   \\ 
& 500 & 42.9 & 76.0 & 89.5 & 92.4 & 92.0 & 90.5 & 76.2 & 87.9   \\ \hline

$\sigma_u=8$ 
& 100 &   9.0 & 20.0 & 28.2 & 33.7 & 33.9 & 34.4 & 23.2 & 27.7   \\ 
& 200 & 25.8 & 46.2 & 61.5 & 64.9 & 65.7 & 65.0 & 52.0 & 67.1   \\ 
& 300 & 44.6 & 71.0 & 83.7 & 87.8 & 89.6 & 89.4 & 72.4 & 87.4   \\ 
& 500 & 77.1 & 93.0 & 97.3 & 97.7 & 98.3 & 98.7 & 92.3 & 98.4   \\ \hline
	               
$\sigma_u=10$ 
& 100 & 11.2 & 25.9 & 35.6 & 39.7 & 41.9 & 41.3 & 28.2 & 32.1  \\ 
& 200 & 30.9 & 51.4 & 68.4 & 72.8 & 74.4 & 75.2 & 57.8 & 71.4  \\ 
& 300 & 52.9 & 79.5 & 90.6 & 91.9 & 93.0 & 93.3 & 84.6 & 93.1  \\ 
& 500 & 83.9 & 93.5 & 98.5 & 98.6 & 98.8 & 98.9 & 96.2 & 99.5  \\ \hline

\multicolumn{10}{l}{\footnotesize Number of Monte Carlo iterations $M=1,000$, KS: Kolmogorov-Smirnov test, CvM: Cram\'er-von Mises test,} \\
\multicolumn{10}{l}{\footnotesize  normal: standard normal, half-normal: ${\rm{HN}}(0,\sigma^2_u)$.} 
\end{tabular}
\end{table*}

In the Table \ref{expsize} the size results (percentage of rejection rounded to the nearest integer) for the tests $T_{n,\lambda}$, KS and CvM are presented at level of significance $\alpha=5\%$, corresponding to the ${\cal{N}}(0,1)/{\rm{Exp}}(\theta)$ null hypothesis. 
Table \ref{exppower1} shows power results for the normal/half--normal alternative hypothesis ${\cal{N}}(0,1)/{\rm{HN}}(0,\sigma^2_u)$. For power results corresponding to some extra simulation settings we refer the interested reader to the accompanying \textit{Supplement}. From Table \ref{expsize} we see that for all  three tests the empirical size varies with the value of the exponential parameter $\theta$, while for the new test $T_{n,\lambda}$ figures also vary with the weight parameter $\lambda$. Overall however and with a few exceptions the nominal size is satisfactorily recovered.  Turning to Table \ref{exppower1} we observe that  the power is low for all tests when the sample size $n$ is small with lower values of the half--normal parameter $\sigma_u$, but progressively increases with $n$ as $\sigma_u$ gets larger, in which case  the new test $T_{n,\lambda}$ enjoys a clear advantage against its competitors, at least for higher values of the weight parameter $\lambda$.

\subsection{Simulations for the normal/gamma case}

In Table \ref{gammasize} we present level results for the test in \eqref{tsgamma} under the normal/gamma null hypothesis with $v_j\sim{\cal{N}}(0,1)$ and $u_j\sim {\rm{Gamma}}(\kappa=2,\theta)$ for the same values of $\theta$ considered in \S 3.1. We note that Stevenson \citeyearpar{St80} was probably the first to consider such a model in the context of stochastic frontiers. As before estimators of parameters were obtained by maximum  likelihood. The results in Table \ref{gammasize} show that the three tests respect the nominal size to a satisfactory degree. The power results are reported in the Tables  \ref{gammapower1}, \ref{gammapower2}, and \ref{gammapower3}, and correspond to powers of the test based on $T_{n,\lambda}$ in \eqref{tsgamma} as well as the KS and CM tests for the null hypothesis normal/gamma with  $\kappa=2$, against the alternatives normal/exponential (i.e. normal/gamma with $\kappa=1$),  normal/gamma with  $\kappa=3$ and  normal/gamma with  $\kappa=0.5$, respectively. The percentage of rejection varies with the alternative under consideration, being relatively low for $\kappa=3$, but increases considerably for the other two alternatives. The message that may be drawn from these results is that the new test with larger values of $\lambda$ ($\lambda=4$ or $5$)  seems to be preferable to its competitors almost uniformly  with respect to  the sample size $n$ and the alternative being considered.

\begin{table*}[h!]    % Table 3
\centering
\small
\setlength{\tabcolsep}{7pt} 
\renewcommand{\arraystretch}{0.90} 
\caption{Size of the test for the normal/gamma null hypothesis with $\kappa=2$, at level of significance $\alpha$ and sample size $n$ }
\label{gammasize}
\begin{tabular}{| l  r | c c c c c c | c | c |}
\hline
$\alpha=5\%$ &  & \multicolumn{6}{|c|}{$T_{n,\lambda}$} & \quad KS \quad & \quad CvM \quad \\
\cline{3-8} & $n$ & $\lambda=0.5$ & $\lambda=1.0$ & $\lambda=2.0$  & $\lambda=3.0$ & $\lambda=4.0$ &$\lambda=5.0$ &  &   \\ \hline \hline

$\theta=0.5$ 
& 100   & 4.6	&	4.7	&	5.2	&	5.1	&	5.8	&	6.5	&	4.3	&	6.0 \\
& 200	& 6.0	&	4.0	&	4.6	&	5.7	&	5.4	&	7.0	&	6.1	&	6.0 \\
& 300	& 6.9	&	3.4	&	4.7	&	5.7	&	6.0	&	5.8	&	6.1	&	4.6 \\
& 500	& 4.4	&	5.6	&	4.2	&	3.9	&	2.9	&	2.3	&	3.2	&	4.0 \\ \hline

$\theta=1$
& 100	& 4.9 &	5.1	&	4.7	&	4.1	&	4.6	&	3.4	&	3.1	&	4.7	\\
& 200	& 5.5 &	6.4	&	5.7	&	3.8	&	3.6	&	2.6	&	7.0	&	4.6	\\
& 300	& 4.3 &	3.6	&	5.1	&	4.3	&	4.0	&	3.2	&	3.1	&	3.3	\\
& 500	& 3.3 &	4.8	&	4.8	&	4.2	&	4.9	&	4.1	&	6.1	&	5.5	\\ \hline

$\theta=3$ 
& 100   & 4.5 & 4.3 &	3.4 &	5.0	&	5.4	&	5.6	&	4.5	&	4.4	\\
& 200	& 5.7 & 5.1 &	5.1	&	4.4	&	4.5	&	4.5	&	5.8	&	5.0	\\
& 300	& 4.4 & 3.6 &	5.6	&	5.0	&	4.4	&	4.5	&	4.7	&	5.0	\\
& 500	& 5.5 & 5.0 &	4.5	&	4.9	&	5.1	&	5.0	&	5.9	&	5.1	\\ \hline

$\theta=5$ 
& 100	& 4.8	& 4.8	& 5.3	& 4.7	& 3.8	& 3.7	& 4.1 & 4.5  \\
& 200	& 3.4	& 4.2	& 3.9	& 4.8	& 4.6	& 4.6	& 4.1 & 4.2  \\
& 300	& 5.5	& 5.1	& 4.6	& 4.8	&  5.4	& 5.3	& 4.1 & 4.5  \\
& 500	& 4.9	& 5.2	& 4.6	& 4.6	& 4.7	& 3.7	& 3.4 & 3.6  \\ \hline

$\theta=8$ 
& 100 	& 6.5	&	5.9	&	4.2	&	4.7	&	4.7	&	3.8	&	5.4	&	5.5  \\
& 200	& 5.0	&	5.4	&	5.3	&	6.4	&	5.2	&	6.3	&	4.4	&	5.1  \\
& 300	& 3.3	&	4.1	&	5.4	&	6.0	&	6.1	&	7.2	&	5.9	&	5.7  \\
& 500	& 4.6	&	4.1	&	3.5	&	3.5	&	4.7	&	4.9	&	3.8	&	3.7  \\ \hline

$\theta=10$ 
& 100	& 5.4	&	4.5	&	5.6	&	4.5	&	4.6	&	4.2	&	5.2	& 5.4	\\
& 200	& 4.8	&	5.8	&	4.4	&	3.9	&	5.2	&	5.8	&	4.4	& 4.6	\\
& 300	& 6.5	&	5.5	&	5.8	&	5.3	&	4.3	&	4.8	&	5.1	& 4.9	\\
& 500	& 4.4	&	4.3	&	3.8	&	4.9	&	6.1	&	5.3	&	4.1	& 5.0	\\ \hline

\multicolumn{10}{l}{\footnotesize Number of Monte Carlo iterations $M=1,000$, KS: Kolmogorov-Smirnov test, CvM: Cram\'er-von Mises test,} \\
\multicolumn{10}{l}{\footnotesize  normal: standard normal, gamma: ${\rm{Gamma}}(\kappa=2,\theta)$.} 
\end{tabular}
\end{table*}

\begin{table*}[h!]    % Table 4
\centering
\small
\setlength{\tabcolsep}{7pt} 
\renewcommand{\arraystretch}{0.90} 
\caption{Power of the test against the normal/exponential alternative at level of significance $\alpha$ and sample size $n$}
\label{gammapower1}
\begin{tabular}{| l  r | c c c c c c | c | c |}
\hline
$\alpha=5\%$ &  & \multicolumn{6}{|c|}{$T_{n,\lambda}$} & \quad KS \quad & \quad CvM \quad \\
\cline{3-8} & $n$ & $\lambda=0.5$ & $\lambda=1.0$ & $\lambda=2.0$  & $\lambda=3.0$ & $\lambda=4.0$ &$\lambda=5.0$ &  &   \\ \hline \hline

$\theta=0.5$ 
& 100	&	3.9	&	5.2	&	6.9	&	6.3	&	6.0	&	6.8	&	5.0	&	5.7	\\
& 200	&	4.6	&	4.4	&	4.5	&	5.3	&	5.4	&	5.6	&	5.3	&	3.8	\\
& 300	&	6.1	&	5.1	&	6.6	&	5.9	&	6.0	&	5.9	&	5.5	&	6.0	\\
& 500	&	2.8	&	4.5	&	6.5	&	6.4	&	7.1	&	7.3	&	6.3	&	5.1	\\ \hline

$\theta=1$ 
& 100	&	5.9	&	4.9	&	7.1	&	6.1	&	5.2	&	4.5	&	5.3	&	4.5	\\
& 200	&	5.2	&	5.1	&	6.7	&	4.8	&	4.1	&	4.3	&	5.1	&	6.5	\\
& 300	&	4.2	&	6.9	&	6.2	&	5.0	&	4.6	&	3.8	&	8.0	&	8.0	\\
& 500	&	5.4	&	4.5	&	6.8	&	7.8	&	7.9	&	7.7	&	5.3	&	6.8	\\ \hline

$\theta=3$ 
& 100	&	6.4	&	6.8	    &	7.7	    &	9.1	    &	11.7	&	15.1	&	12.8	&	12.0	\\
& 200	&	8.9	&	9.7	    &	16.8	&	20.8	&	26.2	&	29.2	&	20.7	&	24.5	\\
& 300	&	6.4	&	13.7	&	25.3	&	31.8	&	40.8	&	42.8	&	34.5	&	37.9	\\
& 500	&	8.9	&	19.7	&	34.3	&	44.8	&	55.1	&	60.2	&	49.1	&	56.7	\\ \hline

$\theta=5$ 
& 100	&	3.6	&	10.0	&	20.6	&	31.6	&	38.2	&	40.9	&	25.1	&	32.3	\\
& 200	&	8.8	&	21.5	&	35.4	&	46.8	&	57.0	&	58.8	&	49.1	&	50.6	\\
& 300	&	9.8	&	29.5	&	51.0	&	64.1	&	72.3	&	76.2	&	61.0	&	70.6	\\
& 500	&	13.5&	42.6	&	71.3	&	85.0	&	89.7	&	91.4	&	81.5	&	88.9	\\ \hline

$\theta=8$ 
& 100	&	9.2	    &	16.9	&	42.9	&	56.5	&	61.8	&	59.8	&	46.8	&	59.8	\\
& 200	&	14.2	&	26.0	&	59.9	&	76.1	&	79.8	&	80.4	&	69.4	&	79.3	\\
& 300	&	21.6	&	46.3	&	80.6	&	91.6	&	93.4	&	92.5	&	86.1	&	92.2	\\
& 500	&	34.6	&	69.0	&	96.5	&	99.0	&	99.4	&	99.2	&	98.4	&	99.4	\\ 
\hline

$\theta=10$ 
& 100	&	10.2	&	19.0	&	58.3	&	67.2	&	68.2	&	68.2	&	58.9	&	65.5	\\
& 200	&	13.7	&	37.5	&	79.0	&	88.3	&	89.0	&	87.5	&	79.1	&	86.0	\\
& 300	&	26.5	&	54.2	&	92.4	&	96.4	&	96.5	&	96.1	&	94.5	&	96.1	\\
& 500	&	45.5	&	81.0	&	99.1	&	99.8	&	99.9	&	99.8	&	99.5	&	100.0	\\ \hline

\multicolumn{10}{l}{\footnotesize Number of Monte Carlo iterations $M=1,000$, KS: Kolmogorov-Smirnov test, CvM: Cram\'er-von Mises test,} \\
\multicolumn{10}{l}{\footnotesize  normal: standard normal, exponential: ${\rm{Exp}}(\theta)$.} 
\end{tabular}
\end{table*}

\begin{table*}[h!]    % Table 5
\centering
\small
\setlength{\tabcolsep}{7pt} 
\renewcommand{\arraystretch}{0.90} 
\caption{Power of the test against the normal/gamma alternative with $\kappa=3$ at level of significance $\alpha$ and sample size $n$}
\label{gammapower2}
\begin{tabular}{| l  r | c c c c c c | c | c |}
\hline
$\alpha=5\%$ &  & \multicolumn{6}{|c|}{$T_{n,\lambda}$} & \quad KS \quad & \quad CvM \quad \\
\cline{3-8} & $n$ & $\lambda=0.5$ & $\lambda=1.0$ & $\lambda=2.0$  & $\lambda=3.0$ & $\lambda=4.0$ &$\lambda=5.0$ &  &   \\ \hline \hline

$\theta=0.5$ 
& 100	&	5.2	&	3.6	&	5.4	&	3.9	&	3.5	&	3.9	&	3.5	&	4.1	\\
& 200	&	4.2	&	3.4	&	5.8	&	4.6	&	4.8	&	4.9	&	4.4	&	5.9	\\
& 300	&	4.2	&	5.4	&	4.8	&	5.6	&	6.3	&	4.9	&	4.5	&	3.1	\\
& 500	&	4.9	&	4.9	&	4.5	&	4.6	&	4.4	&	4.4	&	4.7	&	5.4	\\  \hline

$\theta=1$ 
& 100	&	5.9	&	4.3	&	4.8	&	4.0	&	2.4	&	2.0	&	3.8	&	2.8	\\
& 200	&	5.6	&	5.1	&	4.8	&	3.4	&	3.7	&	3.5	&	4.9	&	5.6	\\
& 300	&	4.6	&	6.8	&	5.6	&	5.3	&	4.1	&	3.0	&	5.4	&	4.5	\\
& 500	&	4.3	&	5.3	&	6.5	&	5.5	&	5.4	&	4.8	&	5.9	&	5.4	\\  \hline
		     
$\theta=3$ 
& 100	&	3.6	&	3.9	&	5.5	    &	5.3	&	6.4	&	6.0	&	6.2	&	5.6	\\
& 200	&	4.4	&	4.9	&	7.0	    &	7.0	&	6.7	&	6.6	&	8.4	&	8.1	\\
& 300	&	5.5	&	4.6	&	6.6	    &	7.7	&	8.6	&	9.2	&	8.6	&	9.1	\\
& 500	&	6.4	&	6.9	&	10.3	&	10.1&	12.5&	14.1&	12.0&	12.8	\\  \hline

$\theta=5$ 
& 100	&	5.5	&	5.4	&	5.7	&	6.3	&	6.4	&	6.2	&	6.8	&	7.0	\\
& 200	&	4.8	&	5.3	&	6.8	&	8.8	&	10.0	&	9.1	&	9.0	&	7.5	\\
& 300	&	6.7	&	7.0	&	7.6	&	10.6	&	12.0	&	13.0	&	10.9	&	10.8	\\
& 500	&	5.3	&	6.1	&	9.3	&	12.2	&	17.3	&	18.3	&	13.9	&	15.2	\\  \hline

$\theta=8$ 
& 100	&	4.5	&	6.0	&	8.7	&	10.9	&	11.8	&	10.9	&	10.6	&	12.0	\\
& 200	&	4.2	&	4.9	&	8.3	&	11.3	&	13.4	&	12.1	&	9.3	    &	11.4	\\
& 300	&	4.2	&	7.2	&	12.3&	14.7	&	17.8	&	19.1	&	14.6	&	15.2	\\
& 500	&	4.3	&	6.0	&	12.6&	16.4	&	20.9	&	23.4	&	15.4	&	16.6	\\  \hline

$\theta=10$ 
& 100	&	7.1	&	5.7	&	8.3	    &	10.1	&	10.0	&	10.5	&	9.4	    &	10.7	\\
& 200	&	4.5	&	4.7	&	9.0	    &	12.5	&	14.9	&	16.1	&	10.2	&	12.1	\\
& 300	&	4.9	&	5.6	&	12.9	&	15.3	&	17.7	&	18.1	&	14.8	&	17.3	\\
& 500	&	6.6	&	7.4	&	12.0	&	16.4	&	21.2	&	22.8	&	18.1	&	18.0	\\  \hline

\multicolumn{10}{l}{\footnotesize Number of Monte Carlo iterations $M=1,000$, KS: Kolmogorov-Smirnov test, CvM: Cram\'er-von Mises test,} \\
\multicolumn{10}{l}{\footnotesize  normal: standard normal, gamma: ${\rm{Gamma}}(\kappa=3,\theta)$.} 
\end{tabular}
\end{table*}

\begin{table*}[h!]    % Table 6
\centering
\small
\setlength{\tabcolsep}{7pt} 
\renewcommand{\arraystretch}{0.90} 
\caption{Power of the test against the normal/gamma alternative with $\kappa=0.5$ at level of significance $\alpha$ and sample size $n$}
\label{gammapower3}
\begin{tabular}{| l  r | c c c c c c | c | c |}
\hline
$\alpha=5\%$ &  & \multicolumn{6}{|c|}{$T_{n,\lambda}$} & \quad KS \quad & \quad CvM \quad \\
\cline{3-8} & $n$ & $\lambda=0.5$ & $\lambda=1.0$ & $\lambda=2.0$  & $\lambda=3.0$ & $\lambda=4.0$ &$\lambda=5.0$ &  &   \\ \hline \hline

$\theta=0.5$ 
& 100	&	6.6	&	4.8	&	4.3	&	7.0	&	7.9	&	8.6	&	5.3	&	4.8	\\
& 200	&	5.5	&	5.6	&	3.7	&	6.8	&	7.9	&	7.9	&	4.5	&	3.7	\\
& 300	&	4.3	&	7.0	&	7.4	&	5.0	&	6.1	&	6.2	&	5.7	&	6.1	\\
& 500	&	4.1	&	4.8	&	5.5	&	4.5	&	3.7	&	4.1	&	6.0	&	4.1	\\  \hline

$\theta=1$     
& 100	&	5.0	&	5.3	&	6.0	&	5.8	&	4.1	&	4.4	&	3.6	&	4.7	\\
& 200	&	6.4	&	6.2	&	10.3&	8.6	&	7.8	&	6.7	&	6.1	&	7.6	\\
& 300	&	4.5	&	4.0	&	6.8	&	8.3	&	7.5	&	6.6	&	6.4	&	6.7	\\
& 500	&	4.5	&	4.7	&	7.8	&	8.3	&	9.5	&	9.1	&	7.0	&	7.9	\\  \hline
		     
$\theta=3$ 
& 100	&	4.8	&	7.3	    &	15.8	&	20.7	&	25.1	&	29.7	&	22.2	&	25.3	\\
& 200	&	6.6	&	16.4	&	41.8	&	49.8	&	58.1	&	62.5	&	47.2	&	55.3	\\
& 300	&	8.6	&	23.3	&	59.4	&	72.9	&	79.9	&	83.5	&	63.7	&	75.7	\\
& 500	&	6.5	&	43.8	&	78.9	&	86.8	&	92.2	&	94.6	&	83.8	&	90.5	\\  \hline

$\theta=5$ 
& 100	&	7.2	    &	25.4	&	44.0	&	59.2	&	63.5	&	68.3	&	57.8	&	64.0	\\
& 200	&	13.5	&	51.5	&	80.1	&	88.3	&	94.2	&	95.2	&	88.5	&	93.1	\\
& 300	&	15.3	&	74.4	&	93.4	&	97.4	&	98.5	&	99.0	&	96.9	&	98.5	\\
& 500	&	27.0	&	94.0	&	99.9	&	100	    &	100	    &	100	    &	100	    &	100	\\  \hline

$\theta=8$ 
& 100	&	17.9	&	52.9	&	80.2	&	87.5	&	91.1	&	92.8	&	89.8    &	91.8	\\
& 200	&	38.3	&	83.5	&	97.4	&	99.5	&	99.7	&	99.8	&	99.3	&	99.8	\\
& 300	&	55.5	&	98.0	&	100	    &	100	    &	100	    &	100	    &	100	    &	100	\\
& 500	&	85.1	&	100	    &	100	    &	100	    &	100	    &	100 	&	100	    &	100	\\  \hline

$\theta=10$ 
& 100	&	26.6	&	67.5	&	91.7	&	95.8	&	97.2	&	97.7	&	95.3	&	97.2	\\
& 200	&	58.1	&	94.6	&	99.7	&	99.9	&	99.9	&	99.9	&	100	    &	99.9	\\
& 300	&	81.3	&	99.5	&	100	    &	100	    &	100	    &	100	    &	100	    &	100	\\
& 500	&	98.6	&	100	    &	100	    &	100	    &	100	    &	100	    &	100	    &	100	\\  \hline

\multicolumn{10}{l}{\footnotesize Number of Monte Carlo iterations $M=1,000$, KS: Kolmogorov-Smirnov test, CvM: Cram\'er-von Mises test,} \\
\multicolumn{10}{l}{\footnotesize  normal: standard normal, gamma: ${\rm{Gamma}}(\kappa=0.5,\theta)$.} 
\end{tabular}
\end{table*}

\newpage

\section{Conclusions}\label{sec_4}
We propose goodness--of--fit tests for the distribution of the composed error  $\varepsilon=v-u$ in stochastic frontier production models. The new test statistics are based on the characteristic function of the composed error term $\varepsilon$ and they are omnibus, i.e. they possess non--negligible power asymptotically for any given alternative under test. Moreover, bootstrap versions of the tests are shown to have competitive power compared to the classical Kolmogorov--Smirnov and Cram\'er-von Mises tests in finite samples.  

We wish to close by stressing the fact that the tests presented herein make use of specific properties of the CF underlying the null hypothesis and as such they are tailored for specific hypotheses under test.  If they are to be modified to apply to other cases such as the popular normal/half--normal or normal/gamma with unspecified shape parameter, or any other specification, then one has to employ alternative properties analogous to \eqref{cf11} and \eqref{cfgamma} that apply  to the CF of the specific distribution under test; see for instance the test for the skew normal distribution suggested by Meintanis \citeyearpar{Mei07} which may be used to test the normal/half-normal SFM. On the other hand there also exists a general formulation for a test statistic based on the CF that may be applied to any specification of the law of the composed error $\varepsilon$. To this end suppose we wish to test a particular specification for the composed error that involves a parameter vector, say $\Theta$, containing regression as well as any distributional parameter of this specification. Then this general formulation is given by  
\[
T_{n,w}=n \int_{-\infty}^\infty |\varphi_n(t)-\varphi_\varepsilon(t;\widehat \Theta)|^2 w(t) dt,
\]
where $\varphi_n(t)=C_n(t)+i S_n(t)$ is the empirical CF  computed from estimated residuals, see for instance below equation \eqref{cf12}, and $\varphi_\varepsilon(t;\Theta)$ is the CF under the null hypothesis, both computed on the basis of the estimator $\widehat \Theta$  of the parameter vector $\Theta$ obtained under the particular parametric specification underlying the null hypothesis.  While tests such as the above have the advantage of full generality, this formulation is based on the premise that the null CF  $\varphi_\varepsilon(t;\Theta)$ is known and has a rather simple expression, so that numerical integration is not necessary. Otherwise tests like the ones defined by \eqref{ts} and \eqref{tsgamma} which are tailored, i.e. they make use of the specific structure of the CF under the null hypothesis, may be preferable, at least from the computational point of view.     

%In closing and while the characteristic--function approach may be convenient for testing such null hypotheses as well as for  testing the symmetry hypothesis of the composed error, which is the starting point when implementing stochastic frontier models, we wish to point out that alternative transforms may also be employed. In fact,  the moment generating function and the Laplace transform may turn out to be more convenient to use in certain cases instead of the characteristic function. In this connection notice that the moment generating function  of the composed error may be written as $M_\varepsilon(t)=M_v(t) L_u(t)$ where $M_Z(t)$ $:=\mathbb E(e^{tZ})$ and $L_Z(t):=\mathbb E_Z(e^{-t Z})$ are the moment generating function and the Laplace transform, respectively, of a random variable $Z$. 

\section*{Appendices}
\appendix
\section{Proof of equations \texorpdfstring{\eqref{ts2}}{ h} and \texorpdfstring{\eqref{tsgamma}}{ h}}
\label{Appendix A}
\renewcommand{\appendixname}{AS}

Starting from equation \eqref{cf12} we obtain
\begin{eqnarray*}
D^2_n(t)&=&S^2_n(t)+t^2 C^2_n(t)+2 t S_n(t)C_n(t) \\ &=&\left(\frac{1}{n} \sum_{j=1}^n \sin(\widehat {\widetilde {\varepsilon}}_j)\right)^2+t^2 \left(\frac{1}{n} \sum_{j=1}^n \cos(\widehat {\widetilde {\varepsilon}}_j)\right)^2 
+2t \left(\frac{1}{n} \sum_{j=1}^n \sin(\widehat {\widetilde {\varepsilon}}_j)\right) \left(\frac{1}{n} \sum_{j=1}^n \cos(\widehat {\widetilde {\varepsilon}}_j)\right) 
\\ &=& \frac{1}{n^2} \sum_{j,k=1}^n \sin(\widehat {\widetilde {\varepsilon}}_j)  \sin(\widehat {\widetilde {\varepsilon}}_k) +\frac{t^2}{n^2} \sum_{j,k=1}^n \cos(\widehat {\widetilde {\varepsilon}}_j)  \cos(\widehat {\widetilde {\varepsilon}}_k) 
+ \frac{2t}{n^2} \sum_{j,k=1}^n \sin(\widehat {\widetilde {\varepsilon}}_j)  \cos(\widehat {\widetilde {\varepsilon}}_k),
\end{eqnarray*}
where we write $\sum_{j,k}$ for the double sum $\sum_j\sum_k$. 
Also recall the trigonometric identities 
\begin{eqnarray*}
\sin z_1 \sin z_2=\frac{1}{2}[\cos(z_1-z_2)-\cos(z_1+z_2)] \\
\cos z_1 \cos z_2=\frac{1}{2}[\cos(z_1-z_2)+\cos(z_1+z_2)] \\
\sin z_1 \cos z_2=\frac{1}{2}[\sin(z_1-z_2)+\sin(z_1+z_2)]
\end{eqnarray*}
Now plug  the above expression for $D_n^2(t)$ into the test statistic \eqref{ts} and substitute the above product formulae, and integrate term-by-term the resulting expression. Then after some grouping  we obtain \eqref{ts2} by making use of the integrals 
\[\int_{-\infty}^\infty \cos (t z)  e^{-\lambda |t|}dt=\frac{2 \lambda}{z^2+\lambda^{2}}, \]
\[\int_{-\infty}^\infty t^2 \cos (t z)  e^{-\lambda |t|}dt=\frac{4 \lambda(\lambda^2-3z^2)}{(z^2+\lambda^{2})^3}, \]
\[\int_{-\infty}^\infty t \sin (t z)  e^{-\lambda |t|}dt=\frac{4z \lambda}{(z^2+\lambda^{2})^2}. \]

Equation \eqref{tsgamma} may be proved by following analogous steps, but we also need the extra integrals
\[\int_{-\infty}^\infty t^4 \cos (t z)  e^{-\lambda |t|}dt=\frac{48\lambda (5z^4-10z^2 \lambda^2+\lambda^4)}{(z^2+\lambda^{2})^5}, \]
\[\int_{-\infty}^\infty t^3 \sin (t z)  e^{-\lambda |t|}dt=\frac{48z \lambda(\lambda^2-z^2)}{(z^2+\lambda^{2})^4}. \]

\section{Proof of equation \texorpdfstring{\eqref{limstat}}{ h}} \label{Appendix B} 
Starting from equation \eqref{cf12} and using $\sin(z)=z-(z^3/3!)+\ldots$ and $\cos(z)=1-(z^2/2!)+\ldots$, we obtain (in increasing powers of $t$)
\begin{eqnarray*}
D_n(t)&=&t\left(\frac{1}{n} \sum_{j=1}^n \widehat {\widetilde {\varepsilon}}_j+1\right)-\frac{t^3}{2!}\left( \frac{1}{3} \frac{1}{n} \sum_{j=1}^n \widehat {\widetilde {\varepsilon}}^3_j+ \frac{1}{n} \sum_{j=1}^n \widehat {\widetilde {\varepsilon}}^2_j  \right) +\ldots \ , 
\end{eqnarray*}
and by squaring   
\begin{eqnarray*}
D^2_n(t)&=&t^2\left(\frac{1}{n} \sum_{j=1}^n \widehat {\widetilde {\varepsilon}}_j+1\right)^2 
-t^4\left( \frac{1}{3} \frac{1}{n} \sum_{j=1}^n \widehat {\widetilde {\varepsilon}}^3_j+ \frac{1}{n} \sum_{j=1}^n \widehat {\widetilde {\varepsilon}}^2_j  \right)\left(\frac{1}{n} \sum_{j=1}^n \widehat {\widetilde {\varepsilon}}_j+1\right) +\ldots \ . 
\end{eqnarray*}
Plugging the above expression in equation \eqref{ts} and integrating term-by-term leads to 
\begin{eqnarray*}
T_{n,w}&=&n \ \bigg{[}\frac{4}{\lambda^3}\left(\frac{1}{n} \sum_{j=1}^n \widehat {\widetilde {\varepsilon}}_j+1\right)^2 
-\frac{48}{\lambda^5}\left(\frac{1}{n} \sum_{j=1}^n \widehat {\widetilde {\varepsilon}}_j+1\right)
\left( \frac{1}{3} \frac{1}{n} \sum_{j=1}^n \widehat {\widetilde {\varepsilon}}^3_j+ \frac{1}{n} \sum_{j=1}^n \widehat {\widetilde {\varepsilon}}^2_j  \right) + \ldots \bigg{]}, \end{eqnarray*}
 and by taking the limit as $\lambda \to \infty$, we readily obtain \eqref{limstat},  where we made use of the integral \[\int_{-\infty}^\infty |t|^m e^{-\lambda |t|}dt=\frac{2 \: m!}{\lambda^{m+1}}, \ m=1,2,\ldots  \ . \]

\end{document}